\documentclass[12pt]{amsart}

\usepackage{amssymb,latexsym,amscd}
\usepackage{amsthm,amsmath,mathrsfs}
\usepackage{youngtab}

\newcommand{\slfr}{\mathfrak{sl}}

\newcommand{\cV}{\mathcal{V}}

\newcommand{\cO}{\mathcal{O}}

\newcommand{\cG}{\mathcal{G}}

\newcommand{\cM}{\mathcal{M}}

\newcommand{\cL}{\mathcal{L}}

\newcommand{\cC}{\mathcal{C}}

\newcommand{\Sym}{\mathrm{Sym}}
\newcommand{\Pic}{\mathrm{Pic}}

\newcommand{\Stab}{\mathrm{Stab}}

\newcommand{\lra}{\longrightarrow}

\newcommand{\ra}{\rightarrow}

\newcommand{\PP}{\mathbb{P}}

\newcommand{\cP}{\mathcal{P}}

\newcommand{\ZZ}{\mathbb{Z}}

\newcommand{\CC}{\mathbb{C}}

\newcommand{\End}{\mathrm{End}}
\newcommand{\Mrm}{\mathrm{M}}
\newcommand{\GG}{\mathrm{G}_2}

\newcommand{\GL}{\mathrm{GL}}

\newcommand{\SL}{\mathrm{SL}}
\newcommand{\SO}{\mathrm{SO}}
\newcommand{\OO}{\mathbb{O}}
\newcommand{\HH}{\mathbb{H}}

\newcommand{\im}{\mathrm{im}}
\newcommand{\IM}{\mathrm{Im}}

\def\map#1{\ \smash{\mathop{\longrightarrow}\limits^{#1}}\ }

\theoremstyle{plain}

\newtheorem{thm}{Theorem}[section]

\newtheorem{lem}[thm]{Lemma}

\newtheorem{prop}[thm]{Proposition}

\newtheorem{cor}[thm]{Corollary}

\begin{document}

\title[]{The space of generalized $\GG$-theta functions of level one}

\author{Chlo\'e Gr\'egoire}
\author{Christian Pauly}

\address{Institut Fourier \\ 100 rue des maths \\ BP74 \\ 38402 St Martin d'H\`eres Cedex \\ France}

\email{chloe.gregoire@gmail.com}

\address{Laboratoire de Math\'{e}matiques J.-A. Dieudonn\'e \\ Universit\'e de Nice - Sophia Antipolis \\
06108 Nice Cedex 02 \\ France}

\email{pauly@unice.fr}


\maketitle

\begin{abstract}
Let $C$ be a smooth projective complex curve of genus at least $2$. 
For a simply-connected complex Lie group $G$ the vector space of global sections 
$H^0(\mathcal{M}(G), \mathcal{L}_G^{\otimes l})$ of the $l$-th power of the ample generator 
$\mathcal{L}_G$ of the Picard group of the
moduli stack of principal $G$-bundles over $C$ is commonly called the space of generalized $G$-theta functions or Verlinde space of level $l$.
In the case $G = \GG$, the exceptional Lie group of automorphisms of the complex Cayley algebra, we study natural linear 
maps between the Verlinde space $H^0(\mathcal{M}(G_2), \mathcal{L}_{\mathrm{G}_2})$ of level one and some Verlinde spaces 
for $\SL_2$ and $\SL_3$. We deduce that the image of the monodromy representation of the WZW-connection for $G = \GG$ and
$l=1$ is infinite.
\end{abstract}

\section{Introduction}
Let $C$ be a smooth projective complex curve of genus $g \geq 2$. For a complex semi-simple Lie group $G$ we denote 
by $\cM(G)$ the moduli stack of principal $G$-bundles over $C$. If $G$ is simply-connected, the Picard group of the stack
$\cM(G)$ is infinite cyclic and we denote by $\cL$ its ample generator. The finite-dimensional vector spaces of global sections 
$H^0(\cM(G), \cL^{\otimes l})$, the so-called spaces of generalized $G$-theta functions or Verlinde spaces of level $l$, have been intensively studied 
from different perspectives, e.g. gauge theory, mathematical theory of conformal blocks, and quantization. Note that much of the literature deals with 
the vector bundle case $G = \SL_r$.

\bigskip

In this note we will study the Verlinde space $H^0(\mathcal{M}(\GG), \mathcal{L}_{\mathrm{G}_2})$ for the smallest exceptional Lie group $\GG$ and
at level $1$. The starting point of our investigation was the striking numerical relation between the dimensions of the Verlinde
spaces for $\GG$ at level $1$ and for $\SL_2$ at level $3$
\begin{equation} \label{Verlinde}
 \dim H^0(\mathcal{M}(\GG), \mathcal{L}_{\mathrm{G}_2}) = \frac{1}{2^g} \dim H^0(\mathcal{M}(\SL_2), \mathcal{L}^{\otimes 3}_{\SL_2}) =
\left( \frac{5 + \sqrt{5}}{2} \right)^{g-1} + \left( \frac{5 - \sqrt{5}}{2} \right)^{g-1}.
\end{equation}
These dimensions are computed by the Verlinde formula (see e.g. \cite{B4} Corollary 9.8). It turns out that linear maps
between these Verlinde spaces arise in a natural way by restricting to some distinguished substacks in $\cM(\GG)$. The group 
$\GG$ contains the subgroups $\SL_3$ and $\SO_4$ as maximal reductive subgroups of maximal rank. These group
inclusions induce maps
$$ i: \cM(\SL_3) \lra \cM(\GG) \qquad \text{and} \qquad j : \cM(\SL_2) \times \cM(\SL_2) \lra  \cM(\GG)$$
via the \'etale double cover $\SL_2 \times \SL_2 \ra \SO_4$. 

\bigskip

Our main results are the following.

\bigskip

\noindent

{\bf Theorem I.}
For any smooth curve $C$ of genus $g \geq 2$ the linear map obtained by pull-back by
the map $j$ of global sections of $\cL_{\mathrm{G}_2}$
$$ j^*:  H^0(\cM(\GG), \cL_{\mathrm{G}_2} )  \lra  \left[ H^0(\cM(\SL_2), \cL_{\SL_2}^{\otimes 3} ) \otimes 
H^0(\cM(\SL_2), \cL_{\SL_2} ) \right]_0$$
is an isomorphism.

\bigskip
\noindent

{\bf Theorem II.}
For any smooth curve $C$ of genus $g \geq 2$ without vanishing theta-null the linear map obtained by pull-back
by the map $i$ of global sections of $\cL_{\mathrm{G}_2}$
$$ i^*:  H^0(\cM(\GG), \cL_{\mathrm{G}_2} ) \lra  H^0(\cM(\SL_3), \cL_{\SL_3})_+ $$
is surjective.

\bigskip
\noindent

The subscripts $0$ and $+$ denote subspaces of invariant sections for the group of $2$-torsion line bundles over $C$ and for
the duality involution respectively.

\bigskip

The first example of isomorphism between Verlinde spaces was given in \cite{B2} for the embedding $\CC^* \subset \SL_2$ at level $1$.
More recently, the rank-level dualities provide series of isomorphisms between Verlinde spaces (and their duals) for special pairs of structure groups.
In this context Theorem I can be viewed as a new example.

\bigskip

Most of the constructions presented in this paper are valid for the coarse moduli spaces of semi-stable $G$-bundles over $C$. However,
the generator $\mathcal{L}_{\mathrm{G}_2}$  of the Picard group of the moduli stack $\cM(\GG)$ does not descend \cite{LS} to the moduli space
$M(\GG)$ because the Dynkin index of $\GG$ is $2$. This forces us to use the moduli stack.

\bigskip

Theorem I has an application to the flat projective connection on the bundle of conformal blocks associated to the Lie algebra $\mathfrak{g}_2$
at level $1$. Let $\pi : \cC \ra S$ be a family of smooth projective curves and consider the vector
bundle $\mathbb{V}^*_1(\mathfrak{g}_2)$ over $S$ whose fiber over the curve $C = \pi^{-1}(s)$ equals the conformal block
$\mathcal{V}_1^*(\mathfrak{g}_2)$. Note that this conformal block is canonically (up to homothety) isomorphic to our space 
$H^0(\cM(\GG), \cL_{\mathrm{G}_2})$ by the general Verlinde isomorphism \cite{LS}. By \cite{U} the vector bundle
$\mathbb{V}^*_1(\mathfrak{g}_2)$ is equipped with a flat projective connection, the so-called WZW-connection. Then we have
the

\bigskip
\noindent

{\bf Corollary.} There exist families of smooth curves of any genus $g \geq 2$ for which the projective monodromy
representation of the projective WZW-connection on $\mathbb{V}^*_1(\mathfrak{g}_2)$ has infinite image.

\bigskip

In section 2 we review the properties of the
exceptional group $\GG$ and of its subgroups, as well as some results on the Verlinde spaces for $\SL_2$ at low levels. In section 3 we 
give the proof of the main theorems.

\bigskip
\noindent 

{\bf Acknowledgements:} Part of the results of this article are contained in the PhD-thesis of the first author. We would
like to thank the referee for useful comments on the first version of this article and in particular for drawing our attention to
the paper \cite{A}. We also thank Laurent Manivel and Olivier Serman for useful discussions during the preparation of this article.

\bigskip

\section{Moduli spaces and moduli stacks of principal $\GG$-bundles}

In this section we review some results on the exceptional group $\GG$ and on the moduli of principal $\GG$-bundles over
a smooth projective curve $C$.

\subsection{The exceptional group $\GG$ and its rank-$2$ subgroups}

The complex exceptional group $\GG$ is given by one of the following equivalent definitions :

\begin{itemize}

\item as the automorphism group $\GG = \mathrm{Aut}(\OO)$ of the complex $8$-dimensional Cayley algebra or algebra of octonions $\OO$ (see e.g. \cite{Ba}).

\item as the connected component of the stabilizer in $\GL(V)$ of a non-degenerate alternating trilinear form $\omega : \Lambda^3 V \ra \CC$ on a 
complex $7$-dimensional vector space $V$ (see e.g. \cite{SK})

\end{itemize}

We recall the following facts:

\begin{enumerate}
\item[(a)] For a generic trilinear form $\omega$ we have  $\Stab_{\mathrm{GL}(V)}(\omega) = \GG \times \mu_3$ and 
$\Stab_{\mathrm{SL}(V)}(\omega) = \GG$. Note that non-degenerate alternating
forms form the unique dense $\GL(V)$-orbit in $\Lambda^3 V^*$.
\item[(b)] Introducing $\GG$ as $\mathrm{Aut}(\OO)$ there is a natural non-degenerate $\GG$-invariant trilinear form on the space of purely
imaginary octonions $V = \IM(\OO)$ given
by $\omega(x,y,z) = \mathrm{Re}(xyz)$, as well as a non-degenerate symmetric  $\GG$-invariant bilinear form given by $q(x,y) = \mathrm{Re}(xy)$. 
This shows that $\GG$ is a subgroup of $\SO_7$.
\item[(c)] The complex Lie group $\GG$ is connected, simply-connected, has no center and is of dimension $14$. 
\end{enumerate}

\medskip
\noindent

According to \cite{BD} the group $\GG$ has, up to conjugation, two maximal Lie subgroups of maximal rank, i.e. of rank $2$. These two subgroups are of type $A_2$ and $A_1 \times A_1$ respectively. As we could not find a reference in the literature, we will give for the reader's convenience an explicit realization of these subgroups in $\GG$.

\begin{itemize}
\item  $\SL_3 \subset \GG$. We consider a non-degenerate alternating trilinear form $\omega \in \Lambda^3 V^*$ and define 
$\GG = \Stab_{\SL(V)}(\omega)$. We associate to $\omega$ the quadratic form
$$ q_\omega : \Sym^2 V \ra \CC, \qquad  q_\omega (x,y) = L_x \omega \wedge L_y \omega \wedge \omega \in \Lambda^7 V^* \cong \CC, $$
where $L_x : \Lambda^3 V^* \ra \Lambda^2 V^*$ denotes the contraction operator with the vector $x \in V$. Note that $\omega$ 
is non-degenerate iff $q_\omega$ is non-degenerate. We now choose a $3$-dimensional subspace $W \subset V$ such that $W$ is isotropic for $q_\omega$
and such that the restriction $\omega_0 = \omega_{| W} \not= 0$. The following gives a description of $\SL_3$ as a subgroup of $\GG$.

\begin{prop}
With the above notation we have
$$ \SL_3 = \Stab_{\mathrm{G}_2} (W) = \{ g \in \GG \ | \ g(W) = W \}.$$
More precisely, the subspace $W \subset V$ induces a natural decomposition 
\begin{equation} \label{decV}
 V = W \oplus \Lambda^2 W \oplus \CC, 
\end{equation} 
which coincides with the decomposition of $V$ as $\SL_3$-module.
\end{prop}

\begin{proof}
We consider the composite map 
$$ \iota : \Lambda^2 W \hookrightarrow \Lambda^2 V \map{L_\omega} V^*, $$
where $L_\omega$ is contraction with $\omega \in \Lambda^3 V^*$. If we further compose with the projection
$V^* \ra W^*$, we obtain the isomorphism $\Lambda^2 W \map{\sim} W^*$ induced by the non-zero restricted form
$\omega_0$.  Hence $\iota$ is injective and we also denote by $\Lambda^2 W  \subset V$ its image in $V$, which 
we identify with $V^*$ via the non-degenerate quadratic form $q_\omega$. Next we observe that $W \cap \Lambda^2 W 
= \{ 0 \}$, since the composite map $W \ra V^* \ra W^*$ is zero --- $W$ is isotropic. This shows that 
$W \oplus \Lambda^2 W$ is a hyperplane in $V$. Then we take the orthogonal complement to obtain the 
decomposition \eqref{decV}. We observe that any $g \in \Stab_{\mathrm{G}_2} (W)$ also preserves the subspace
$\Lambda^2 W  \subset V$, hence the decomposition \eqref{decV}. Moreover, since $g(\omega_0) = \omega_0$, we have
$g \in \SL_3 = \SL(W)$. Hence  $\Stab_{\mathrm{G}_2} (W) \subset \SL_3$. On the other hand we consider the action of
$\GG$ on the Grassmannian of isotropic subspaces $W \subset V$, which is of dimension $6$. Hence $\dim 
\Stab_{\mathrm{G}_2}(W) \geq 8$, which leads to the equality $\Stab_{\mathrm{G}_2} (W) = \SL_3$. 
\end{proof}

\bigskip

\item  $\SO_4 \subset \GG$. We need to recall some basic facts on quarternions and octonions. We begin by recalling that 
the complex octonion algebra $\OO$ is generated as $\CC$-vector space by the $8$ basis
vectors $e_0 = 1 , e_1, \ldots, e_7$ satisfying the relations given by the Fano plane (see e.g. \cite{Ba}). Then the algebra $\OO$ contains 
as a subalgebra the complex quaternion algebra $\HH = \CC 1  \oplus \CC e_1 \oplus \CC e_2 \oplus \CC e_3$ and we have a vector space
decomposition 
\begin{equation} \label{decOO}
 \OO = \HH \oplus \HH e_4.
\end{equation}
We recall that the subgroup $U = \{ p \in \HH \ | \ p \overline{p} = 1 \}$ of unit quarternions can be identified with the complex Lie 
group $\SL_2$ and that there is a surjective group homomorphism
$$\varphi : U \times U \lra \SO(\HH) = \SO_4, \qquad \varphi(p,q) = [x \mapsto px\overline{q}] $$
with kernel $\ZZ/2$ generated by $(-1,-1)$. Using the decomposition \eqref{decOO} we consider the map
$$\psi:  U \times U \lra \SO(\OO),  \qquad \psi(p,q) = (\varphi(p,p), \varphi(p,q)).$$
One easily checks that $\im \ \psi \subset \GG$ and that $\ker \psi = \ker \varphi$. This gives a realization of $\SO_4$ as
subgroup of $\GG$. We also note that the center $Z(\mathrm{SO}_4)$ is generated by $\varphi(-1,1) = - \mathrm{Id}_\HH$ and that $\SO_4$ 
is the centralizer of the element $\psi(-1,1) = (\mathrm{Id}_\HH, - \mathrm{Id}_\HH) \in \GG$ of order $2$ (see \cite{BD}).

\end{itemize}

\subsection{The moduli space $\Mrm(\GG)$ and the moduli stack $\cM(\GG)$}

Because of the equality $\Stab_{\mathrm{SL}(V)}(\omega) = \GG$, a principal $\GG$-bundle $E_{\mathrm{G}_2}$ is equivalent to a rank-$7$
vector bundle $\mathcal{V}$ with trivial determinant  equipped with a non-degenerate alternating trilinear from $\eta : 
\Lambda^3 \mathcal{V} \ra \cO_C$ . The correspondence
is given by sending $E_{\mathrm{G}_2}$ to  $(\mathcal{V} = E_{\mathrm{G}_2} (V), \eta)$ via the embedding $\GG \subset \SL(V)$.
Moreover, it is shown in \cite{S} that $E_{\mathrm{G}_2}$ is semi-stable if and only if $\mathcal{V}$ is semi-stable. We therefore
obtain a map between coarse moduli spaces of semi-stable bundles $\Mrm(\GG) \ra \Mrm(\SL_7)$.

\bigskip

Although the embeddings of $\SL_3$ and $\SO_4$ in $\GG$ are defined only up to conjugation, the induced maps between coarse moduli spaces
of semi-stable principal bundles 
$$i: \Mrm(\SL_3) \ra  \Mrm(\GG) \qquad  \text{and} \qquad j : \Mrm(\SL_2) \times \Mrm(\SL_2) \ra \Mrm(\SO_4) \ra \Mrm(\GG)$$
are well-defined. We find it more convenient to work with the simply-connected group $\SL_2 \times \SL_2$, which is a double cover of the
subgroup $\SO_4$.  Abusing notation we also denote by $i$ and $j$ their composites with the map $\Mrm(\GG) \ra \Mrm(\SL_7)$. It follows from
the description of the subgroups $\SL_3$ and $\SO_4$ in the previous section that
\begin{equation} \label{desj}
i(E) = E \oplus E^* \oplus \cO_C \qquad \text{and} \qquad  j(F, G)  = \End_0(F) \oplus F \otimes G.
\end{equation}
Here $E$ is an $\SL_3$-bundle and $F$, $G$ are $\SL_2$-bundles. Note that $i(E)$ and $j(F,G)$ are semi-stable if $E,F$ and $G$ are semi-stable.

\bigskip

{\bf Remark:} It is shown in \cite{G} that the singular locus of the moduli space $\Mrm(\GG)$ coincides with the union of the
images $i(\Mrm(\SL_3)) \cup j(\Mrm(\SO_4))$.

\bigskip

We also denote by $i$ and $j$ the maps between the corresponding moduli stacks. Let $\cL_G$ denote the ample generator of the Picard group 
$\Pic(\cM(G))$ when $G$ is a simply-connected group. 

\begin{lem}
With the above notation we have
$$i^* \cL_{\mathrm{G}_2} = \cL_{\SL_3} \qquad \text{and} \qquad  j^* \cL_{\mathrm{G}_2}  = \cL_{\SL_2}^{\otimes 3} \boxtimes \cL_{\SL_2}.$$
\end{lem} 

\begin{proof}
This follows straightforwardly from a Dynkin index computation using the tables in \cite{LS}.
\end{proof}

\bigskip

We consider the involution $\sigma : \cM(\SL_3) \ra \cM(\SL_3)$ given by taking the dual $\sigma(E) = E^*$. Then the line 
bundle $\cL_{\SL_3}$ is invariant under the involution $\sigma$. We consider the linearisation 
$\sigma^* \cL_{\SL_3} \map{\sim} \cL_{\SL_3}$ which restricts to the identity over the fixed points of $\sigma$ and
denote by $H^0(\cM(\SL_3), \cL_{\SL_3})_+$ the subspace of invariant sections.

\bigskip

The group of $2$-torsion line bundles $JC[2]$ acts on $\cM(\SL_2)$ by tensor product and the Mumford group 
$\mathcal{G}(\cL_{\SL_2})$, a central extension of $JC[2]$, acts linearly on $H^0(\cM(\SL_2), \cL_{\SL_2})$ with level $1$.
The $\mathcal{G}(\cL_{\SL_2})$-representation  $H^0(\cM(\SL_2), \cL_{\SL_2}^{\otimes 3} ) \otimes 
H^0(\cM(\SL_2), \cL_{\SL_2} )$ is of level $4$ and therefore admits a linear $JC[2]$-action.

\begin{prop}
The induced maps between Verlinde spaces 
\begin{eqnarray*}
i^* : H^0(\cM(\GG), \cL_{\mathrm{G}_2} ) & \lra & H^0(\cM(\SL_3), \cL_{\SL_3})_+ \\
j^* : H^0(\cM(\GG), \cL_{\mathrm{G}_2} ) & \lra & \left[ H^0(\cM(\SL_2), \cL_{\SL_2}^{\otimes 3} ) \otimes 
H^0(\cM(\SL_2), \cL_{\SL_2} ) \right]_0
\end{eqnarray*}
take values in the subspace invariant under the involution $\sigma$ and the $JC[2]$-action respectively.

\end{prop}

\begin{proof}
First we show that the map $i : \cM(\SL_3) \ra \cM(\GG)$ is $\sigma$-invariant. There is a natural inclusion
between Weyl groups $W(\SL_3) \subset W(\GG)$. Consider an element $g \in \GG$ which lifts an element in
$W(\GG) \setminus W(\SL_3)$. Then $g \notin \SL_3$.
As the subalgebra $\mathfrak{sl}_3$ of $\mathfrak{g}_2$ corresponds to the long roots and as $W(\GG)$ preserves the 
Cartan-Killing form, the inner automorphism $C(g)$ of $\GG$ induced by $g$ preserves the subgroup $\SL_3$. The restriction
of $C(g)$ to $\SL_3$ is an outer automorphism, which permutes its two fundamental representations. It thus induces 
the involution $\sigma$ on the moduli stack $\cM(\SL_3)$. Since any inner automorphism of $\GG$ induces the identity 
on the moduli stack $\cM(\GG)$, we obtain that $i$ is $\sigma$-invariant.

\bigskip

Since $i^* \cL_{\mathrm{G}_2} = \cL_{\SL_3}$ and since $i$ is $\sigma$-invariant, the line bundle $\cL_{\SL_3}$ carries a
natural $\sigma$-linearisation, namely the one which restricts to the identity over fixed points of $\sigma$. It is now clear that
$\im (i^*) \subset  H^0(\cM(\SL_3), \cL_{\SL_3})_+$.

\bigskip

The second statement follows immediately from the fact that $j$ is invariant under the diagonal $JC[2]$-action on the moduli stack
$\cM(\SL_2) \times \cM(\SL_2)$.
\end{proof}

\bigskip

\subsection{A family of divisors in $\PP H^0(\mathcal{M}(\GG), \mathcal{L}_{\mathrm{G}_2})$}

Let $\theta(C)$ resp. $\theta^+(C)$ denote the set of theta-characteristics resp. even theta-characteristics over the curve $C$.
The moduli stack $\cM(\SO_7)$ has two connected components $\cM^+(\SO_7)$ and $\cM^-(\SO_7)$ distinguished by the second Stiefel-Whitney class.
Since $\cM(\GG)$ is connected, the homomorphism $\GG \subset \SO_7$ induces a map
$$ \rho : \cM(\GG) \ra  \cM^+(\SO_7).$$
For each $\kappa \in \theta(C)$ we introduce the Pfaffian line bundle $\cP_\kappa$ over $\cM^+(\SO_7)$ (see e.g. \cite{BLS} section 5). We
have 
$$ \rho^* \cP_\kappa = \cL_{\mathrm{G}_2}. $$
Moreover, for $\kappa \in \theta^+(C)$, there exists a Cartier divisor $\Delta_\kappa \in \PP H^0(\cM^+(\SO_7), \cP_\kappa)$ with support
$$ \mathrm{supp} (\Delta_\kappa) = \{ E \in \cM^+(\SO_7) \ | \ \dim H^0(C, E(\CC^7) \otimes \kappa ) > 0 \}, $$
where $E(\CC^7)$ denotes the rank-$7$ vector bundle associated to $E$. We also denote by $\Delta_\kappa \in 
\PP H^0(\mathcal{M}(\GG), \mathcal{L}_{\mathrm{G}_2})$ the pull-back $\rho^*(\Delta_\kappa)$ to  $\cM(\GG)$. 
We will show later (Corollary \ref{spandelta}) that the family of divisors $\{ \Delta_\kappa \} _{\kappa \in \theta^+(C)}$ spans
the linear system $\PP H^0(\mathcal{M}(\GG), \mathcal{L}_{\mathrm{G}_2})$. Abusing notation we also write $\Delta_\kappa$ for a 
section of $H^0(\mathcal{M}(\GG), \mathcal{L}_{\mathrm{G}_2})$ vanishing at the divisor $\Delta_\kappa$.

\bigskip

\subsection{Verlinde spaces for $\SL_2$ at level $1$, $2$ and $3$}

We denote $V_n = H^0(\cM(\SL_2), \cL_{SL_2}^{\otimes n})$ for $n \geq 1$. We now review some results from \cite{B1} describing 
special bases of the vector spaces $V_1 \otimes V_1$ and $V_2$.

\bigskip

First we recall that the Mumford group $\cG(\cL_{\SL_2})$ acts linearly on the space $V_n$ with level $n$, i.e., the center $\CC^*
\subset \cG(\cL_{\SL_2})$ acts via $\lambda \mapsto \lambda^n$. For $n$ odd, there exists a unique (up to isomorphism) irreducible 
$\cG(\cL_{\SL_2})$-module $H_n$ of level $n$. Note that $\dim H_n = 2^g$. If $n$ is divisible by $4$, any $\cG(\cL_{\SL_2})$-module $Z$
of level $n$ admits a linear $JC[2]$-action. We denote by $Z_0$ the $JC[2]$-invariant subspace of $Z$.

\bigskip

We now list the results needed in the proof of Theorem II.

\bigskip

\begin{lem} \label{diminv}
We have 
$$ \dim \left( V_1 \otimes V_3 \right)_0 =  \frac{1}{|JC[2]|} \dim V_1 \otimes V_3. $$
\end{lem}

\begin{proof}
By the general representation theory of Heisenberg groups, the $\cG( \cL_{\SL_2})$-module $V_1 \otimes V_3$ 
decomposes into a direct sum of factors, which are all isomorphic to $H_1 \otimes H_3$. It is then straightforward to show that 
the space of $JC[2]$-invariants
$\left( H_1 \otimes H_3 \right)_0$ is $1$-dimensional.
\end{proof}

\bigskip

\begin{prop}[\cite{B1}] \label{beauresult}
The two $\cG(\cL_{\SL_2})$-modules $V_1 \otimes V_1$ and $V_2$ of level $2$ decompose as direct sum of $1$-dimensional 
character spaces for $\cG(\cL_{\SL_2})$
$$ V_1 \otimes V_1 = \bigoplus_{\kappa \in \theta(C)} \CC \xi_\kappa, \qquad
V_2 = \bigoplus_{\kappa \in \theta^+(C)} \CC d_\kappa. $$  
The supports of the zero divisors $Z(d_\kappa)$ and $Z(\xi_\kappa)$ equal
\begin{eqnarray*} 
\mathrm{supp}  \ Z(d_\kappa) & = & \{ E \in \cM(\SL_2) \ | \ \dim H^0(C, \End_0(E) \otimes \kappa ) > 0 \}, \\
\mathrm{supp}  \ Z(\xi_\kappa) & = & \{ (E,F) \in \cM(\SL_2) \times \cM(\SL_2) \ | \ \dim H^0(C, E \otimes F \otimes \kappa ) > 0 \}.
\end{eqnarray*}
Moreover, if $C$ has no vanishing theta-null, $\xi_\kappa$ is mapped to $d_\kappa$ by the multiplication map $V_1 \otimes V_1 \ra V_2$.  

\end{prop}

\begin{prop}[\cite{A}] \label{aberesult}
For a general curve the multiplication map of global sections 
$$ \mu : V_1 \otimes V_2 \lra V_3 $$
is surjective.
\end{prop}

\bigskip

\section{Proof of the Main Theorems}

In this section we give the proof of the two theorems and of the corollary stated in the introduction.

\subsection{Theorem I}
The first step is to show that the two spaces appearing at both ends of the map $j^*$ have the same dimension. The dimension of
the space on the RHS is computed by Lemma \ref{diminv}. The statement then follows from  \eqref{Verlinde} and 
from the equalities $\dim V_1 = 2^g$ and $|JC[2]| = 2^{2g}$. 

\bigskip

The next step is to show that $j^*$ is surjective for a {\em general} curve, which will imply by the first step that $j^*$ 
is an isomorphism for a general curve. We consider the following map 
$$ \alpha : V_1 \otimes V_1 \otimes V_2 \lra V_1 \otimes V_3, \qquad u \otimes v \otimes w \mapsto u \otimes \mu(v \otimes w), $$
where $\mu$ is the multiplication map introduced in Proposition \ref{aberesult}. By Proposition \ref{aberesult} $\alpha$ is
surjective for a general curve, hence its restriction to the subspace of $JC[2]$-invariant sections $\alpha_0 : (V_1 \otimes V_1 \otimes V_2)_0  \lra
 (V_1 \otimes V_3)_0$ remains surjective. Moreover, one easily works out that the family of tensors $\{ \xi_\kappa \otimes d_\kappa \}_{ \kappa
 \in \theta^+(C) }$ forms a basis of $(V_1 \otimes V_1 \otimes V_2)_0$ .
 
\bigskip

We will use the family of divisors $\{ \Delta_\kappa \}_{\kappa \in \theta^+(C)}$ introduced in section 2.3.

\begin{lem} \label{imageDeltakappa}
For all $\kappa \in \theta^+(C)$ we have the equality (up to a scalar)
$$ j^*(\Delta_\kappa) = \alpha_0 ( \xi_\kappa \otimes d_\kappa ).$$
\end{lem}

\begin{proof}
Using the description of $j$ given in \eqref{desj} and the description of the divisors $Z(d_\kappa)$ and $Z(\xi_\kappa)$ given
in Proposition \ref{beauresult} we obtain the following decomposition as divisor in $\cM(\SL_2) \times  \cM(\SL_2)$
$$ j^*(\Delta_\kappa) =  pr_1^* Z(d_\kappa) + Z(\xi_\kappa),$$
where $pr_1$ is the projection onto the first factor. This shows the lemma.
\end{proof}

Surjectivity (for a general curve) now follows as follows: since $\{ \xi_\kappa \otimes d_\kappa \}_{ \kappa
 \in \theta^+(C) }$ forms a basis of $(V_1 \otimes V_1 \otimes V_2)_0$  and since $\alpha_0$ is surjective, we see by Lemma \ref{imageDeltakappa} 
 that the family $\{ j^*(\Delta_\kappa) \}_{ \kappa \in \theta^+(C) }$ generates $(V_1 \otimes V_3)_0$.
 
\bigskip

Finally, we complete the proof by showing that $j^*$ is an isomorphism for every smooth curve. We use the identification \cite{LS} for any semi-simple,
simply-connected complex Lie group $G$ of the Verlinde space $H^0(\cM(G), \cL_G^{\otimes l})$ with the space of conformal blocks 
$\cV^*_l(\mathfrak{g})$ at level $l$, where $\mathfrak{g}$ is the Lie algebra of $G$, for the two cases $G = \GG$ and $G = \SL_2 \times \SL_2$.
Then, \cite{Be} Proposition 5.2 shows functoriality of the above isomorphism under group extensions. In our case $\SL_2 \times \SL_2 \ra \GG$
the linear map $j^*$ can therfore be identified with the natural map
$$ \beta_C : \cV^*_1(\mathfrak{g}_2) \lra  \cV^*_3(\mathfrak{sl}_2) \otimes \cV^*_1(\mathfrak{sl}_2).$$
We can define this linear map for a family of smooth curves $\pi : \cC \ra S$:  by \cite{U} there exist vector bundles of conformal blocks
over the base $S$ together with a homomorphism $\beta$, which specializes over a point $s \in S$  to the linear map $\beta_{\pi^{-1}(s)}$. These
vector bundles are equipped with flat projective connections, the so-called WZW connections.

\bigskip

We now observe that the Lie algebra embedding $\mathfrak{sl}_2 \oplus \mathfrak{sl}_2 \subset \mathfrak{g}_2$ is conformal --- by 
direct computation. Then, by \cite{Be} Proposition 5.8, we obtain that the map $\beta$ is projectively flat for the two WZW connections, hence
its rank is constant in the family $\pi: \cC \ra S$. Since, by the previous step, $\beta_C$ is injective for a general curve $C$ (note that we do not take 
$JC[2]$-invariants on the conformal blocks), we conclude that $\beta$ is injective for any smooth curve. Hence $j^*$ is an 
isomorphism for any curve. This completes the proof.

\bigskip

From the above proof we immeadiately deduce the 

\begin{cor} \label{spandelta}
For a general curve the family $\{ \Delta_\kappa \}_{\kappa \in \theta^+(C)}$ linearly spans $\PP H^0(\mathcal{M}(\GG), \mathcal{L}_{\mathrm{G}_2})$.
\end{cor}

\bigskip

{\bf Remark: } Note that Hitchin's connection \cite{H1} is only defined on the vector bundle with fiber
 $H^0(\cM(\GG), \cL_{\mathrm{G}_2}^{\otimes 2})$. We only get a connection for $\GG$ at level $1$ via the
 isomorphism with the bundle of conformal blocks.

 \bigskip

\subsection{Theorem II}
We consider the family of divisors $\{ \Delta_\kappa \} _{\kappa \in \theta^+(C)}$ of $\PP H^0(\mathcal{M}(\GG), \mathcal{L}_{\mathrm{G}_2})$
introduced in section 2.3. A straightforward computation shows that $i^*(\Delta_\kappa) = H_\kappa$, where $H_\kappa \in 
\PP H^0(\cM(\SL_3), \cL)_+$ is the divisor with support 
$$ \mathrm{supp} (H_\kappa) = \{ E \in \cM(\SL_3) \ | \ \dim H^0(C, E \otimes \kappa ) > 0 \}. $$
Therefore, in order to show surjectivity of $i^*$, it is enough to show that the family $\{ H_\kappa \}_{\kappa \in \theta^+(C)}$ linearly 
spans $\PP H^0(\cM(\SL_3), \cL)_+$. This is done as follows.

\bigskip

We introduce the Riemann Theta divisor $\Theta = \{ L \in \Pic^{g-1}(C) \ |  \  \dim H^0( C, L) > 0 \}$ in the Picard variety $\Pic^{g-1}(C)$ parameterizing degree $g-1$ line bundles over $C$. We recall \cite{BNR} that there is a canonical isomorphism
\begin{equation} \label{bnr3}
H^0(\Pic^{g-1}(C), 3\Theta)^* \map{\sim} H^0(\cM(\SL_3), \cL),
\end{equation}
which is invariant for the two involutions, i.e. $L \mapsto K_C \otimes L^{-1}$ on $\Pic^{g-1}(C)$ and $\sigma$ on 
$\cM(\SL_3)$ respectively. We thus obtain an isomorphism between subspaces of invariant divisors 
$|3\Theta|_+^* \cong \PP H^0(\cM(\SL_3), \cL)_+$. Is is easy to check that via this isomorphism 
$H_\kappa = \varphi_{3\Theta}(\kappa)$, where 
$$ \varphi_{3\Theta} : \Pic^{g-1}(C) \dasharrow |3 \Theta|_+^*$$
is the rational map given by the linear system $|3\Theta|_+$. In order to show that the family of points 
$\{ \varphi_{3\Theta}(\kappa) \}_{\kappa \in \theta^+(C)}$ linearly spans $|3\Theta|^*_+$, we factorize the map $\varphi_{3\Theta}$ as
$$\varphi_{4\Theta} : \Pic^{g-1}(C) \dasharrow |4\Theta|^*_+ \dasharrow |3\Theta|^*_+,$$
where the first map is the rational map given by the linear system $|4\Theta|^*_+$ and the second is the 
projection induced by the inclusion $H^0(3\Theta)_+ \map{+ \Theta} H^0(4\Theta)_+$. The result then follows from
the main statement in \cite{KPS} saying that $\{ \varphi_{4\Theta}(\kappa) \}_{\kappa \in \theta^+(C)}$ is a projective
basis of $|4\Theta|^*_+$ if $C$ has no vanishing theta-null.

\bigskip

{\bf Remark:} For a curve of genus $2$ we observe that both spaces have the same dimension, hence $i^*$ is an isomorphism in that case --- note that
any genus-$2$ curve is without vanishing theta-null.

\bigskip

\subsection{Corollary}
The statement of the corollary is proved in \cite{LPS} for the conformal block $\cV^*_3(\mathfrak{sl}_2) =  H^0(\cM(\SL_2), \cL_{SL_2}^{\otimes 3})$.
Having already observed in the proof of Theorem I that the vector bundle map $\beta$ is projectively flat for the WZW-connections, it suffices
to show the statement for the $JC[2]$-invariants of $\cV_3^*(\mathfrak{sl}_2) \otimes \cV_1^*(\mathfrak{sl}_2)$. This follows from
\cite{Be} Corollary 4.2.

\bigskip

\section{Some remarks}

Here we collect some additional computations.

\subsection{Verlinde formula for $l=2$ and $g=2$} We just record the computation of the Verlinde number $\dim H^0(\cM(\GG), \cL^2) = 30$. Since 
the line bundle $\cL^2$ descends to the coarse moduli space $\Mrm(\GG)$ we obtain a rational $\theta$-map 
$$ \theta : \Mrm(\GG) \lra |\cL^2|^* = \PP^{29}. $$
We refer to the paper \cite{B3} for some results on  the $\theta$-map on a genus-$2$ curve for vector bundles of small rank.

\subsection{Analogue for the exceptional group $\mathrm{F}_4$} There is a similar coincidence for the conformal embedding of Lie algebras
$\slfr(2) \oplus \mathfrak{sp}(6) \subset \mathfrak{f}_4$. In fact, we observe that 
$\dim H^0(\cM(\mathrm{F}_4), \cL_{\mathrm{F}_4}) = \dim H^0(\mathcal{M}(\GG), \mathcal{L}_{\mathrm{G}_2})$
and that $\dim H^0(\cM(\mathrm{Sp}_6), \cL_{\mathrm{Sp}_6}) = \dim H^0(\mathcal{M}(\SL_2), \mathcal{L}^{\otimes 3}_{\SL_2})$ (symplectic
strange duality). Moreover, $\ker ( \SL_2 \times \mathrm{Sp}_6 \ra \mathrm{F}_4) = \ZZ/2$. These facts suggest a similar isomorphism 
for the Verlinde space $H^0(\cM(\mathrm{F}_4), \cL_{\mathrm{F}_4})$, but the method presented in this paper does to apply to that case.

\end{document}